%% file: paper.tex
\documentclass[a4paper,12pt,titlepage=false]{scrartcl}

\usepackage{bbm}
\usepackage{amsmath}
\usepackage{amssymb}
\usepackage{amsfonts}
\usepackage{amsthm}
\usepackage{amscd}
\usepackage[english]{babel}
\usepackage[latin1]{inputenc}
\usepackage[T1]{fontenc}
\usepackage{lmodern}
\usepackage{fancyhdr, fancybox}
\usepackage{extramarks}
\usepackage{array}
\usepackage{multirow}
\usepackage[hang,symbol]{footmisc}
\usepackage{color}
\usepackage{graphicx}
\usepackage{caption}
\usepackage{nicefrac}
\usepackage{enumerate} 

\usepackage{tikz}
\usetikzlibrary{matrix,arrows,positioning}
\usepackage{hyperref}
\usepackage[intoc]{nomencl}
\usepackage{hyphsubst}
\usepackage{makeidx}
\usepackage{relsize}

\usepackage{ellipsis}

\hypersetup{pdfborder={0 0 0}}

\newtheorem{thm}{Theorem}[section]
\newtheorem{lem}[thm]{Lemma}
\newtheorem{cor}[thm]{Corollary}
\newtheorem{prop}[thm]{Proposition}
\newtheorem*{thm*}{Theorem}

\theoremstyle{definition}
\newtheorem{example}[thm]{Example}

\newtheorem{rmk}[thm]{Remark}

\usepackage{environ}
\NewEnviron{hiddenproofs}{}

\newcommand{\ses}[5]{ %
0 \to #1 \overset{#4}{\longrightarrow} #2 \overset{#5}{\longrightarrow} #3 \to 0 %
}

\renewcommand{\phi}{\varphi}

\DeclareMathOperator{\im}{im}

\DeclareMathOperator{\rad}{Rad}

\DeclareMathOperator{\soc}{Soc}

\DeclareMathOperator{\module}{mod}
\renewcommand{\mod}{\module}

\DeclareMathOperator{\Hom}{Hom}
\DeclareMathOperator{\Aut}{Aut}
\DeclareMathOperator{\Ext}{Ext}

\DeclareMathOperator{\aut}{Aut}
\DeclareMathOperator{\res}{res}
\DeclareMathOperator{\ind}{ind}
\DeclareMathOperator{\Char}{char}
\DeclareMathOperator{\End}{End}

\DeclareMathOperator{\ql}{ql}

\DeclareMathOperator{\Lie}{Lie}

\DeclareMathOperator{\sl2}{\mathfrak{sl}(2)}
\DeclareMathOperator{\Usl2}{U_0(\mathfrak{sl}(2))}

\DeclareMathOperator{\cx}{cx}

\newcommand{\N}{\mathbb{N}}
\newcommand{\Z}{\mathbb{Z}}
\renewcommand{\P}{\mathbb{P}}

\newcommand{\G}{\mathcal{G}}
\newcommand{\cN}{\mathcal{N}}
\newcommand{\cM}{\mathcal{M}}
\newcommand{\cH}{\mathcal{H}}

\newcommand{\cg}{\mathfrak{g}}
\newcommand{\cC}{\mathcal{C}}
\newcommand{\cS}{\mathcal{S}}
\newcommand{\cQ}{\mathcal{Q}}

\newcommand{\hideproofs}{\let\proof\hiddenproofs
	\let\endproof\endhiddenproofs}

\title{\large{Classification of indecomposable modules for finite group schemes of domestic representation type}}
\author{\small Dirk Kirchhoff}
\date{}

\begin{document}
\maketitle

\input{intro}
\input{sl2trmodules}
\input{filt_ind_mod}
\input{modsl2_1}
\input{quot_sup_var}
\input{decomp_ind_mod}
\input{ar_group_graded}
\input{mod_dom_group_scheme}
\input{acknowledgement}

\markboth{\nomname}{\nomname}

\footnotesize
\bibliographystyle{abbrv}
\bibliography{paper}
\noindent
\textsc{Christian-Albrechts-Universität zu Kiel, Ludewig-Meyn-Str. 4, 24098 Kiel, Germany}
\textit{E-mail address:} \textsf{kirchhoff@math.uni-kiel.de}
\end{document}

%% file: intro.tex
\begin{abstract}\footnotesize{
\textsc{Abstract.}
We investigate the representation theory of domestic group schemes  $\G$ over an algebraically closed
field of characteristic $p>2$.
We present results about filtrations of induced modules, actions on support varieties, Clifford theory for
certain group schemes and applications of Clifford theory for strongly group graded algebras 
to the structure of Auslander-Reiten quivers. The combination of these results leads to the classification
of modules belonging to the principal blocks of the group algebra $k\G$.
}\end{abstract}

\section*{Introduction}

In representation theory we have a trichotomy of representation types. Over an algebraically closed field any algebra $A$
has either finite, tame or wild representation type. \\
We say that $A$ has finite representation type if $A$ possesses only finitely many isomorphism classes of indecomposable modules.
An algebra $A$ has tame representation type if it is not of finite representation type and if in each dimension almost all isomorphism classes of 
indecomposable modules occur in only finitely many one-parameter families. \\
The representation type of group algebras of finite groups was determined in \cite{bondarenko1982representation}. 
Let $k$ be a field of characteristic $p>0$ and $G$ be a $p$-group. Then the group algebra $kG$ has tame representation type
if and only if $p=2$ and $G$ is a dihedral, semidihedral or generalized quaternion group.
The classification of modules for a tame algebra can be a hard endeavour. For example up to now there is no such classification
for the quaternion group $Q_8$.\\
Another class of examples are the tame hereditary algebras (\cite{dlab1976indecomposable}). These algebras have the additional property, that the number
of one-parameter families is uniformly bounded. In general an algebra with this property is called of domestic representation type.
The only $p$-group with domestic group algebra is the Klein four group. Its representation theory is clearly
related to that of the 2-Kronecker quiver.\\
In the setting of group algebras for finite group schemes there occur more domestic group algebras.
We will call a finite group scheme $\G$ domestic if its group algebra $k\G:=(k[\G])^*$ is domestic. \\
In \cite{farnsteiner2012extensions} Farnsteiner classified the finite domestic group schemes 
over a field of characteristic $p>2$. Let $\G$ be a finite domestic group scheme. Then there is a tame
hereditary algebra which is Morita equivalent to all blocks of $k\G$.
Moreover, the principal blocks of these group schemes are isomorphic to the principal blocks of
certain domestic group schemes, the so-called amalgamated polyhedral group schemes.\\
The goal of this work is the classification of the indecomposable modules
of the amalgamated polyhedral group schemes. A foundation for this is Premets work (\cite{premet1991green}) on the representation theory of the restricted
Lie algebra $\sl2$. Farnsteiner started in \cite{farnsteiner2009group} to extend these results to the infinitesimal case,
the group schemes $SL(2)_1T_r$ for $r\geq 1$. 
These results will be summarized in the first section.
The missing part was a realization of the periodic $SL(2)_1T_r$-modules.
This gap will be closed in section \ref{sec: realizing modules in tubes}:\\

Let $\G$ be a finite group scheme and $\cN$ a normal subgroup scheme of $\G$ such that $\G/\cN$
is infinitesimal. For an $\cN$-module $Z$ Voigt \cite{voigt1977induzierte} introduced a filtration 
\[Z=N_0\subseteq N_1\subseteq N_2\subseteq\dots\subseteq N_{n-1}\subseteq k\G\otimes_{k\cN}Z\]
by $\cN$-modules and used it to give a generalized version of Clifford theory in form of a
splitting criterion of the short exact sequences
\[\ses{N_{l-1}}{N_l}{N_l/N_{l-1}}{}{}.\]
We develop in section \ref{sec: filtrations of induced modules} a criterion which ensures that none of these sequences split. 
For certain quasi-simple modules $Z$ this result implies that all constituents of the filtration belong to the same AR-component.
In section \ref{sec: realizing modules in tubes} we show when these assumptions are true for modules over $SL(2)_1$. Therefore we obtain new
realizations of those modules and are able to show when they have an $SL(2)_1T_r$-module structure.
These are exactly those modules which were missing in \cite{farnsteiner2009group}. \\
We then turn to a different topic in section \ref{sec:action on rank varieties}. If $\G$ is a finite
group scheme then $G:=\G(k)$ acts on the projectivized rank variety $\P(V_\cg)$ 
where $\cg$ denotes the restricted Lie algebra of $\G$.
This action gives nice properties for the stabilizers of periodic modules. 
If the action of $G$ on $\P(V_\cg)$ is faithful and the variety $\P(V_\cg)$ is
smooth and irreducible, then the stabilizer $G_M$ of any periodic module $M$ is contained in $GL_r(k)$, where
$r$ is the dimension of $\P(V_\cg)$. Especially if the connected component $\G^0$ of $\G$ 
is tame we obtain that the stabilizer $G_M$ is cyclic.\\
The goal of section \ref{sec:decomp induced module} is to prove a generalized Clifford theory decomposition result of induced modules
for certain group schemes. For this we need a normal subgroup scheme $\cN$ of a finite group scheme $\G$
which is contained in $\G^0$ such that $\G^0/\cN$ is multiplicative.
The indecomposable $\cN$-modules in consideration need to be restrictions of $\G$-modules and under the assumption of an additional
stability criterion the decomposition of the induced module corresponds to the decomposition of
$k(\G/\cN)$ into projective indecomposable modules.\\
In section \ref{sec:AR-quiver group graded} we pick up results of \cite{marcus1999representation} about the application of Clifford theory
over strongly group graded algebras to Auslander-Reiten quivers. We analyse the effects of the restriction
functor between components of the occurring Auslander-Reiten quivers for cyclic groups.
Especially if the components are tubes, we can give a relation between their ranks. \\
The last section combines the results of prior sections to describe
the structure of the amalgamated polyhedral group schemes. As an example, we will analyse the
representation theory of the amalgamated dihedral group schemes. The classification of modules for the
remaining amalgamated polyhedral group schemes can now be done with the same methods.\\
\\
All modules and algebras occurring in this work are supposed to be finite-dimensional
over $k$.

%% file: sl2trmodules.tex
\section{The modules and Auslander-Reiten quiver of $SL(2)_1T_r$}
Let $k$ be an algebraically closed field of characteristic $p>2$. The group algebra $kSL(2)_1:=(k[SL(2)_1])^*$ (the dual
of the coordinate ring $k[SL(2)_1]$ ) is isomorphic to the restricted universal 
enveloping algebra $\Usl2$ of the restricted Lie algebra $\sl2$. There are one-to-one correspondences between the representations of $SL(2)_1$, $\Usl2$ and
the restricted representations of $\sl2$. The indecomposable representations of the restricted Lie algebra $\sl2$ were classified by Premet in \cite{premet1991green}.
In \cite[4.1]{farnsteiner2009group} Farnsteiner incorporated these results into the Auslander-Reiten theory of this algebra. Let $T\subseteq SL(2)$
be the standard torus of diagonal matrices. Following \cite[4.1]{farnsteiner2009group},
we will give here an overview of the representation theory of the group schemes $SL(2)_1T_r$, which is based on Premet's work. \\
Let $\{e,f,h\}$ denote the standard basis of $\sl2$.
For $d\in\N_0$ we consider the $(d+1)$-dimensional Weyl module $V(d)$ of highest weight $d$. These are rational $SL(2)$-modules which are obtained
by twisting the $2$-dimensional standard module with the Cartan involution ($x\mapsto -x^{tr}$) and taking its $d$-th symmetric power.
For $d\leq p-1$ we obtain in this way exactly the simple $\Usl2$-modules.
For $s\in\N$, $a\in\{0,\dots,p-2\}$, and $d=sp+a$ Premet introduced the maximal $\Usl2$-submodule $W(d)$ of $V(d)$ of dimension $sp$. These modules are stable 
under the action of the standard Borel subgroup $B\subseteq SL(2)$ of upper triangular matrices. \\
The group $SL(2,k)$ operates on $\Usl2$ via the adjoint representation and for each element $g\in SL(2,k)$ the space $g.W(d)$ is a $\Usl2$-module which is isomorphic
to $W(d)^g$, the space $W(d)$ with action twisted by $g^{-1}$. \\

Let $(\mathfrak{g},[p])$ be a restricted Lie algebra. 
We denote by  
 $V_{\mathfrak{g}}=\{x\in\mathfrak{g}\vert\: x^{[p]}=0\}$
the nullcone of $\mathfrak{g}$. For any $x\in V_{\mathfrak{g}}$ the algebra $U_0(kx)$ is a subalgebra of $U_0(\mathfrak{g})$. 
For any $U_0(\mathfrak{g})$-module $M$ we define its rank variety by
\[V_{\mathfrak{g}}(M):=\{x\in V_{\mathfrak{g}}\:\vert\: M{\vert_{U_0(kx)}} \:\text{is not projective}\}\cup\{0\}.\]
  The dimension of this rank variety is equal to the complexity $\cx_\cg(M)$ of the module $M$ (\cite{friedlander1986support}).
  For $\mathfrak{g}=\sl2$ the nullcone is given by
 \[V_{\sl2}=\{\left(\begin{smallmatrix}
 a & b \\
 c & -a
 \end{smallmatrix}\right)\vert a^2+bc=0\}\]
and for all $g\in SL(2,k)$ we have $V_{\sl2}(g.W(d))=k(geg^{-1})$.\\
 
Let $A$ be a self-injective $k$-algebra. We denote by $\Gamma_s(A)$ the stable Auslander-Reiten quiver of $A$.
The vertices of this valued quiver are the isomorphism classes of non-projective indecomposable $A$-modules and the arrows correspond
to the irreducible morphisms between these modules. Moreover, we have an automorphism $\tau_A$ of $\Gamma_s(A)$, called the Auslander-Reiten translation.
For a self-injective algebra the Auslander-Reiten translation is the composite $\nu\circ\Omega^2$, where $\nu$ denotes the Nakayama functor of $\mod A$ and
$\Omega$ the Heller shift of $\mod A$. For further details we refer to \cite{auslander1997representation} and \cite{assem2006elements}.\\
Let $\Theta\subseteq\Gamma_s(A)$ be a connected component. Thanks to the Struktursatz of Riedtmann \cite{riedtmann1980algebren}, there
is an isomorphism of stable translation quivers $\Theta\cong\Z[T_\Theta]/\Pi$ where $T_\Theta$ denotes a directed tree and $\Pi$ is an admissible 
subgroup of $\Aut(\Z[T_\Theta])$. The underlying undirected tree $\overline{T}_\Theta$ is called the tree class of $\Theta$.
If $\Theta$ has tree class $A_\infty$, then there is for each module $M$ only one sectional path to the end of the component
(\cite[(VII)]{auslander1997representation}). The length of this path is called the quasi-length $\ql(M)$ of $M$.
The modules of quasi-length $1$ are also called quasi-simple. Components of the form $\Z[A_\infty]/(\tau^n)$, $n\geq 1$,
are called tubes of rank $n$. These components contain for each $l\geq1$ exactly $n$ modules of quasi-length $l$.
Tubes of rank $1$ are also called homogeneous tubes and all other tubes are called exceptional tubes.\\
 
Let $\mathfrak{b}$ be the Borel subalgebra of $\sl2$ which is generated by $h$ and $e$. For each $i\in\{0,\dots,p-1\}$ let $k_i$ be the one-dimensional $U_0(\mathfrak{b})$-module
with $h.1=i$ and $e.1=0$. Then the induced $\Usl2$-module $Z(i):=\Usl2\otimes_{U_0(\mathfrak{b})}k_i$ is called a baby Verma module of highest weight $i$.
\begin{lem}\cite[4.1.2]{farnsteiner2009group}
Let $s\in\N$, $a\in\{0,\dots,p-2\}, d=sp+a$ and $g\in SL(2,k)$. Then the AR-component $\Theta\subseteq\Gamma_s(\sl2)$ containing $g.W(d)$ is a homogeneous tube
with quasi-simple module $Z(a)^g$. Moreover, we have $\ql(g.W(d))=s$.
\end{lem}

The Auslander-Reiten quiver of each non-simple block of $kSL(2)_1T_r$ consists of two components of
type $\Z[\tilde{A}_{p^{r-1},p^{r-1}}]$, four exceptional tubes $\Z[A_\infty]/(\tau^{p^{r-1}})$
and infinitely many homogeneous tubes $\Z[A_\infty]/(\tau)$.\\
Denote by $w_0:=\left(\begin{smallmatrix}
0 & 1 \\
-1 & 0
\end{smallmatrix}\right)$ the standard generator of the Weyl group of $SL(2)$.
The spaces $W(sp+a)$ and $w_0.W(sp+a)$ are stable under the action of $SL(2)_1T_r$ and therefore already $SL(2)_1T_r$-modules.
These modules (and certain twists of them) belong to the exceptional tubes. There was no realization given in
 \cite{farnsteiner2009group}  of the modules belonging to homogeneous tubes. But the following was shown:
\begin{lem}\cite[4.3]{farnsteiner2009group}\label{X(i,g,l)}
For each $l\in\N$, $i\in\{0,\dots,p-2\}$ and $g\in SL(2)\setminus(B\cup w_0B)$ there is, up to isomorphism,
 a unique indecomposable $SL(2)_1T_r$-module $X(i,g,l)$ 
with $\res_{SL(2)_1}X(i,g,l)\cong g.W(lp^r+i)$.
\end{lem}
We will see in section \ref{sec: realizing modules in tubes} how to realize these modules.
The $SL(2)_1T_r$-modules are then classified in the following way:
\begin{thm}\cite[4.3.1]{farnsteiner2009group}
Let $C\subseteq SL(2,k)$ be a set of representatives of $SL(2,k)/B$ such that $\{1,w_0\}\subseteq C$ and $M$ be a non-projective indecomposable $SL(2)_1T_r$-module.
Then $M$ is isomorphic to a module of the following list of pairwise non-isomorphic $SL(2)_1T_r$-modules:
\begin{itemize}
\item $V(d)\otimes_k k_\lambda, V(d)^*\otimes_k k_\lambda$, $V(i)\otimes_k k_\lambda$ for $d\geq p$, $\lambda\in X(\mu_{(p^{r-1})})$,
 $d\:{\not\equiv}-1\;(\mod\;p)$ and $0\leq i\leq p-1$.
	 (Modules belonging to Euclidean components)
\item $w_0^j.W(d)\otimes_k k_\lambda$ for $j\in\{0,1\}$, $d=sp+a$ with $a\in\{0,\dots,p-2\}$, $s\in\N$ and $\lambda\in X(\mu_{(p^{r-1})})$.
	 (Modules belonging to exceptional tubes)
\item $X(i,g,l)$ for $g\in C\setminus\{1,w_0\}$ and $d=sp+a$ with $l\in\N$ and $i\in\{0,\dots,p-2\}$.
	 (Modules belonging to homogeneous tubes)
\end{itemize}
\end{thm}

%% file: filt_ind_mod.tex
\section{Filtrations of induced modules}\label{sec: filtrations of induced modules}
Let $k$ be a field of characteristic $p>0$, $\mathcal{G}$ a finite group scheme and $\mathcal{N}\subseteq\mathcal{G}$ a normal subgroup scheme
such that $\mathcal{G}/\mathcal{N}$ is infinitesimal.
Denote by $k\G:=k[\G]^*$ the group algebra of $\G$.
Let $J$ be the kernel of the canonical projection $k[\mathcal{G}]\rightarrow k[\mathcal{N}]$.
The algebra $k[\mathcal{G}]^{\mathcal{N}}\cong k[\G/\cN]$ (see \cite[I.5.5(6)]{jantzen2007representations}) is local and consequently the ideal $I:=k[\mathcal{G}]^{\mathcal{N}}\cap J$ is nilpotent. 
Moreover, $I$ is the augmentation ideal of $k[\G]^\cN$ and therefore $J=Ik[\G]$ by 
\cite[2.1]{waterhouse1979introduction}.
Hence $J$ is also nilpotent.
Thus setting $H_l:=(J^{l+1})^\perp=\{v\in k\G\:\vert\: v(J^{l+1})=(0)\}$ gives us an ascending filtration of $k\G$ consisting of 
$(k\mathcal{N},k\mathcal{N})$-bimodules
\[(0)=H_{-1}\subseteq k\mathcal{N}=H_0\subseteq H_1\subseteq\dots\subseteq H_n=k\G.\]
Now let $Z$ be an $\mathcal{N}$-module. 
Due to \cite[9.5]{voigt1977induzierte}, the canonical maps 
$\iota_l:H_l\otimes_{k\mathcal{N}}Z\rightarrow k\G\otimes_{k\mathcal{N}}Z$
are injective. Set $N_l:=\im\iota_l$. In \cite[9]{voigt1977induzierte} Voigt introduced the following
ascending filtration by  $\cN$-modules of the $\G$-module $N_{n}=k\G\otimes_{k\mathcal{N}}Z$:
\[(0)=N_{-1}\subseteq Z\cong N_0\subseteq N_1\subseteq\dots\subseteq N_{n}.\]
The algebra $k\G$ becomes a $(k[\G]^\cN,k\cN)$-bimodule via $(x.h)(y)=h(yx)$ and $h\bullet h'=hh'$ for all
$h\in k\G, h'\in k\cN, x\in k[\G]^\cN$ and $y\in k[\G]$. Hence the induced module $k\G\otimes_{k\cN} Z$ has
also a $k[\G]^\cN$-module structure. Voigt has given an alternative description of the modules
occurring in the above filtration:
\begin{prop}\label{filtration of induced}\cite[9.6]{voigt1977induzierte}
	In the above situation we get the following equality:
	\[N_l=\{n\in N_n\: \vert\:\forall f\in I^{l+1}: f.n=0\}.\]
	Moreover, the $\cN$-module $N_l/N_{l-1}$ is isomorphic to a direct sum of $\dim_k H_l/H_{l-1}$ copies of $Z$.
\end{prop}	
	Let $f_1,\dots,f_{q_l}$ be generators of the $k[\G]^\cN$-ideal $I^{l+1}$ and 
	\[v_l:k\G\rightarrow(k\G)^{q_l}, h\mapsto (f_1.h,\dots,f_{q_l}.h).\]
	By the proof of \cite[9.6]{voigt1977induzierte}, the map $u_l:=v_l\otimes id_Z:N_n\rightarrow N_n^{q_l}$
	is $\cN$-linear and has kernel $N_l$.\\
	For $1\leq j\leq n$ we define the $\cN$-linear maps $p_{l,j}:={u_l}_{\vert N_j}$. Note that these maps depend on
	the choice of the generators $f_1,\dots,f_{q_l}$. In the case that $I$ is a principal ideal, we fix a generator $f$ of $I$
	and will always choose $f^{l+1}$ as the generator of $I^{l+1}$.
\begin{prop}\label{filtration of induced with one gen}
    Assume that $I$ is a principal ideal. Then the following holds:
    \begin{enumerate}[(a)]
    \item $u_l$ is an $\cN$-linear endomorphism of $N_n$ with $\ker u_l=N_l$,
    \item the dimension of $N_l$ is equal to $(l+1)\dim_kZ$,
    \item $\im p_{l,j}=N_{j-l-1}$ for $1\leq l\leq j\leq n$ and
    \item $p_{m,i}\circ p_{l,j} = p_{m+l+1,j}$ for all $1\leq j\leq i\leq n$.
    \end{enumerate}
\end{prop}    
\begin{proof}
	Since $I$ is a principal ideal the same holds for $I^{l+1}$.
	Therefore $u_l$ is an $\cN$-linear endomorphism of $N_n$, so that (a) holds.\\
	Due to \ref{filtration of induced}, the image of the restriction ${u_l}_{\vert N_{l+1}}$ must lie in $N_0\cong Z$.
	Hence the $\cN$-module $N_{l+1}/N_l$ is isomorphic to a submodule of $Z$. But by
	\ref{filtration of induced} it is also isomorphic to a non-zero direct sum of copies of $Z$. Therefore it must be
	isomorphic to $Z$, which yields (b).
	For $l\leq j\leq n$ another application of \ref{filtration of induced} yields 
	$N_j/N_l\cong\im p_{l,j}\subseteq N_{j-l-1}$, with equality due to dimension reasons.\\
	To show (d) that holds, we first note that $p_{m,i}\circ p_{l,j}=p_{m,j}\circ p_{l,j}$. Now consider the map $v_l:k\G\rightarrow k\G, h\mapsto f^{l+1}.h$,
	where $f$ is the generator of $I$. Then we obtain $v_m\circ v_l(h)=f^{m+l+2}.h=v_{m+l+1}(h)$ for all $h\in k\G$. This yields 
	$p_{m,j}\circ p_{l,j}=p_{m+l+1,j}$.
\end{proof}
Voigt also gave a generalized version of Clifford theory for the decomposition of an induced module (\cite[9.9]{voigt1977induzierte}):
\begin{rmk}\label{ses for filtration of induced}
	The stabilizer $\G_Z$ of $Z$ (see \cite[1.3]{voigt1977induzierte}) equals $\G$ if and only if for all $l\in\{0,\dots,n\}$
	the short exact sequence
	\[\ses{N_{l-1}}{N_l}{N_l/N_{l-1}}{}{}{}\]
	splits.
\end{rmk}
The modules of our interest are in a somewhat opposite situation. We are interested in conditions, when
none of these sequences split.\\
We say that for a $k$-Algebra $A$ an $A$-module $M$ is a brick, if $\End_A(M)\cong k$.
\begin{prop}\label{ses for filtration of unstable module}
	Assume that the following conditions hold:
	\begin{enumerate}[(i)]
	\item $I$ is a principal ideal,
	\item $\dim_k\Ext_{\cN}^1(Z,Z)=1$ and
	\item $k\G\otimes_{k\mathcal{N}}Z$ is a brick.
	\end{enumerate}
	Then for all $l\in\{1,\dots,n\}$
	the short exact sequence
	\[\ses{N_{l-1}}{N_l}{Z}{}{}{}\]
	does not split.
\end{prop}
\begin{proof}
		Since $N_n=k\G\otimes_{k\mathcal{N}}Z$ is a brick, it is indecomposable and the sequence
				\[\ses{N_{n-1}}{N_n}{Z}{}{p_{n-1,n}}\]
		cannot split.
		Hence there is a minimal $l\in\{1,\dots,n\}$ such that the short exact sequence
		\[\ses{N_{l-1}}{N_l}{Z}{}{p_{l-1,l}}\]
		does not split. Assume $l>1$. Then the diagram with exact rows \\
		\begin{center}
			\begin{tikzpicture}
			\matrix (m) [matrix of math nodes, row sep=3em, column sep=3em, text height=1.5ex, text depth=0.25ex]
			{ 0 & N_{l-1}  & N_{l}  & Z  & 0 \\
				0 & N_{l-2} & N_{l-1} & Z & 0 \\ };
			\path[-stealth]
			(m-1-1) edge (m-1-2)
			(m-1-2) edge (m-1-3)
			(m-1-3) edge node [above] {$p_{l-1,l}$} (m-1-4)
			(m-1-4) edge (m-1-5)
			(m-2-1) edge (m-2-2)
			(m-2-2) edge (m-2-3)
			(m-2-3) edge node [below] {$p_{l-2,l-1}$} (m-2-4)
			(m-2-4) edge (m-2-5)
			(m-1-2) edge node [right] {$p_{0,l-1}$} (m-2-2)
			(m-1-3) edge node [right] {$p_{0,l}$}(m-2-3)
			(m-1-4) edge node [right] {$id$}(m-2-4);
			
			\end{tikzpicture}
		\end{center}
		is commutative. If we identify the rows with elements in $\Ext_{\cN}^1(Z,N_{l-1})$
		and $\Ext_{\cN}^1(Z,N_{l-2})$, then the map $p_{0,{l-1}}^*:\Ext_{\cN}^1(Z,N_{l-1})\rightarrow\Ext_{\cN}^1(Z,N_{l-2})$
		sends the first row to the second row (\cite[7.2]{rotman2008introduction}).\\
		By assumption (iii) and Frobenius reciprocity we have 
		\[1\leq\dim_k\Hom_{\cN}(Z,Z)\leq\dim_k\Hom_\cN(Z,k\G\otimes_{k\cN}Z)=\dim_k\End_\G(k\G\otimes_{k\cN}Z)=1.\]
		As $\Hom_\cN(Z,-)$ is left exact, the spaces 
		$\Hom_{\cN}(Z,N_{l-1})$ and $\Hom_{\cN}(Z,N_{l-2})$ can be identified with subspaces of
		$\Hom_{\cN}(Z,k\G\otimes_{k\mathcal{N}}Z)$. As $l>1$ they are non-trivial and consequently also one-dimensional.
		By assumption (ii) we have $\dim_k\Ext_{\cN}^1(Z,Z)=1$.
		Therefore the short exact sequence
		\[\ses{Z}{N_{l-1}}{N_{l-2}}{}{p_{0,l-1}}\]
		induces the long exact sequence
		\[0\rightarrow\Hom_{\cN}(Z,Z)
		\overset{\sim}{\longrightarrow}\Hom_{\cN}(Z,N_{l-1})
		\overset{0}{\longrightarrow}\Hom_{\cN}(Z,N_{l-2})\]
		\[\overset{\sim}{\longrightarrow}\Ext_{\cN}^1(Z,Z)
		\overset{0}{\longrightarrow}\Ext_{\cN}^1(Z,N_{l-1})
		\overset{p_{0,{l-1}}^*}{\longrightarrow}\Ext_{\cN}^1(Z,N_{l-2}).\]
		Hence $p_{0,{l-1}}^*$ is injective and sends non-split exact sequences to non-split exact sequences (\cite[7.2]{rotman2008introduction}.
		Thus the short exact sequence
		\[\ses{N_{l-2}}{N_{l-1}}{Z}{}{p_{l-2,l-1}}\]
		does not split, a contradiction. Consequently $l=1$.
\end{proof}
Before we proof the next result concerning this filtration we need the following:
\begin{lem}\label{homogeneous tube has no projective middle terms}
	Let $\G$ be an infinitesimal group scheme of height $1$, $M$ a $\G$-module which belongs to a homogeneous tube and $\mathcal{E}:\ses{M}{E}{M}{}{}$ the almost split sequence ending in $M$.
	Then $E$ possesses no non-zero projective direct summand.
\end{lem}
\begin{proof}
	Let $B$ be the block of $M$ and assume that $E$ has a non-zero projective indecomposable direct summand $P$. By \cite[IV.3.11]{assem2006elements}
	the sequence $\mathcal{E}$ is equivalent to $\ses{\rad(P)}{\rad(P)/\soc(P)\oplus P}{P/\soc(P)}{}{}$.
	We obtain an isomorphism $M \cong P/\soc(P)$ and therefore $\cx_B(P/\soc(P))=\cx_B(M)=1$.
	Let $(P_i)_{i\geq0}$ be a projective resolution of $\soc(P)$ and set $Q_i:=P_{i+1}, Q_0=P$.
	Then $(Q_i)_{i\geq0}$ is a projective resolution of $P/\soc(P)$. Therefore the simple module $\soc(P)$ has complexity $1$.
	Now \cite[3.2(2)]{farnsteiner1995periodicity} yields that $B$ is a Nakayama algebra and therefore representation finite.
	Hence $M$ belongs to a finite component, a contradiction.
\end{proof}
The following gives us a tool for realizing modules belonging to homogeneous tubes. The assumptions
are  for example fulfilled for $SL(2)_1T_r$.
\begin{prop}\label{filtration of induced quasi simple modules}
	Assume that $\cN$ is an infinitesimal group scheme of height $1$ and that the following conditions hold:
	\begin{enumerate}[(a)]
	\item $I$ is a principal ideal,
	\item $k\G\otimes_{k\mathcal{N}}Z$ is a brick and
	\item  $Z$ is the quasi simple module of a homogeneous tube $\Theta$
			of the stable Auslander-Reiten quiver $\Gamma_s(\cN)$.
	\end{enumerate}
		Then $N_l$ is the indecomposable $\cN$-module of quasi-length $l+1$ in $\Theta$.
\end{prop}
\begin{proof}
Define for all  $l,j\in\{1,\dots,n\}$ with $j\geq l$ the maps
		$\delta_l:=\sum_{i=0}^{l-1}p_{i,l}:N_l\rightarrow N_{l-1}$
		and the injections $\iota_{l,j}:N_l\rightarrow N_j$. Then we get:
		\[p_{0,l+1}\circ\delta_l=\sum_{i=0}^{l-1}p_{0,l+1}\circ p_{i,l}
		\underset{\ref{filtration of induced with one gen}(d)}{=}\sum_{i=0}^{l-1}p_{i+1,l}=\sum_{i=1}^{l}p_{i,l}=\sum_{i=1}^{l-1}p_{i,l}
		=\delta_l-p_{0,l}.\]
		This gives us
		\[
		\left(\begin{smallmatrix}
		\iota_{l-1,l} \:,&
		p_{0,l+1}
		\end{smallmatrix}\right)\circ
		\left(\begin{smallmatrix}
		-\delta_l \\ \iota_{l,l+1}+\delta_l
		\end{smallmatrix}\right)
		=-\delta_l+p_{0,l}+p_{0,l+1}\circ\delta_l=0.
		\]
		Therefore we obtain a short exact sequence:
		
		\begin{center}
			\begin{tikzpicture}
			\matrix (m) [matrix of math nodes, row sep=3em, column sep=3em, text height=1.5ex, text depth=0.25ex]
			{ 
				0 & N_{l} & N_{l-1}\oplus N_{l+1} & N_l & 0. \\ };
			\path[-stealth]
			(m-1-1) edge (m-1-2)
			(m-1-2) edge node [above] {$	\left(\begin{smallmatrix}
				-\delta_l \\ \iota_{l,l+1}+\delta_l
				\end{smallmatrix}\right)$}(m-1-3)
			(m-1-3) edge node [above] {$	\left(\begin{smallmatrix}
				\iota_{l-1,l} \;,&
				p_{0,l+1}
				\end{smallmatrix}\right)$} (m-1-4)
			(m-1-4) edge (m-1-5);
			
			\end{tikzpicture}
		\end{center}
		Now we show by induction over $l$ that $N_l$ belongs to $\Theta$ and has quasi-length $l+1$.
	By \cite[V.2.4]{auslander1997representation}, every exact sequence
	\[\ses{Z}{X}{Z}{}{}\]
	is either split or almost split. Hence we have $\dim_k\Ext_{\cN}^1(Z,Z)=1$.
		Thanks to \ref{ses for filtration of unstable module}, the short exact sequence
		\[\ses{N_{j-1}}{N_j}{Z}{}{p_{j-1,j}}\]
		does not split for all $j\in\{1,\dots,n\}$. Especially the exact sequence
		\[\ses{Z}{N_1}{Z}{}{}\] does not split and therefore is almost split.
		As $Z$ is the quasi-simple module in $\Theta$ and since by \ref{homogeneous tube has no projective middle terms} the middle term of the
		above sequence has no non-zero projective direct summand, it follows that $N_1$ is the indecomposable $\cN$-module
		of quasi-length 2 in $\Theta$.\\
		Now let $l\geq 1$ and assume for all $j\leq l$ that $N_j$ is a module of 
		quasi-length $j+1$ in $\Theta$. As $N_l$ and $N_{l-1}$ are indecomposable $\cN$-module which are not isomorphic to each other the exact sequence

		\begin{center}
			\begin{tikzpicture}
			\matrix (m) [matrix of math nodes, row sep=3em, column sep=3em, text height=1.5ex, text depth=0.25ex]
			{ 
				0 & N_{l} & N_{l-1}\oplus N_{l+1} & N_l & 0 \\ };
			\path[-stealth]
			(m-1-1) edge (m-1-2)
			(m-1-2) edge node [above] {$	\left(\begin{smallmatrix}
				-\delta_l \\ \iota_{l,l+1}+\delta_l
				\end{smallmatrix}\right)$}(m-1-3)
			(m-1-3) edge node [above] {$	\left(\begin{smallmatrix}
				\iota_{l-1,l} \;,&
				p_{0,l+1}
				\end{smallmatrix}\right)$} (m-1-4)
			(m-1-4) edge (m-1-5);
			
			\end{tikzpicture}
		\end{center}
		cannot split. Applying standard properties of almost split sequences (\cite[V.1]{auslander1997representation})
		we obtain $\dim_k\Hom_{\cN}(N_l,N_l)=l+1$.
		As for all $-1\leq i\leq l-1$ the map $\iota_{l-i-1,l}\circ p_{i,l}$ with image $N_{l-i-1}$ belongs to 
		$\Hom_{\cN}(N_l,N_l)$, we get that these maps form a $k$-basis of $\Hom_{\cN}(N_l,N_l)$.
		The only isomorphism of these maps is $\iota_{l,l}\circ p_{-1,l}=id_{N_l}$.
		Hence if $\phi=\sum_{i=-1}^{l-1}\lambda_i\:\iota_{l-i-1,l}\circ p_{i,l}\in\Hom_{\cN}(N_l,N_l)$ is not an isomorphism,
		then $\lambda_{-1}=0$.
		Thus the image of $\phi$ must be a submodule of $N_{l-1}$. But then $\phi$ factors through
		$	\left(\begin{smallmatrix}
		\iota_{l-1,l} \\
		p_{0,l+1}
		\end{smallmatrix}\right)$ and by \cite[V.2.2]{auslander1997representation} the above exact sequence is almost split.
		Moreover, by \ref{homogeneous tube has no projective middle terms} the middle term of this sequence has no non-zero projective
		direct summand.
		Therefore $N_{l+1}$ must be a successor of $N_l$ in $\Theta$.
		As $\Theta$ is a homogeneous tube, the module $N_l$ of quasi-length $l+1$ has exactly two successors, 
		one of quasi-length $l$ and one of quasi-length $l+2$. Since $N_{l-1}$ has quasi-length $l$ it follows
		that $N_{l+1}$ must be the indecomposable $\cN$-module of quasi-length $l+2$ in $\Theta$.
\end{proof}

%% file: modsl2_1.tex
\section{Realizations of periodic $SL(2)_1T_r$-modules}\label{sec: realizing modules in tubes}
Let $k$ be an algebraically closed field of characteristic $p>2$, $T\subseteq SL(2)$ be the torus of diagonal matrices and $B\subseteq SL(2)$ the standard Borel
subgroup of upper triangular matrices. Let $C\subseteq SL(2,k)$ be a set of representatives for $SL(2,k)/B$ with $\{1,w_0\}\subseteq C$
and $g\in C\setminus\{1,w_0\}$. Set $\G:=SL(2)_1T_r$ for $r\geq 1$ and $\cN:=SL(2)_1$. 
For $0\leq a\leq p-2$ we consider the filtration by $\cN$-modules
\[Z(a)^{g}\cong N_0\subseteq N_2\subseteq\dots\subseteq N_{p^{r-1}-1}=k\G\otimes_{k\cN}Z(a)^{g}\]
of the induced module $k\G\otimes_{k\cN}Z(a)^{g}$.
\begin{prop}\label{new realization of sl(2) modules}
	For all $l\in\{0,\dots,p^{r-1}-1\}$, the $\cN$-module $N_l$ is isomorphic to $g.W((l+1)p+a)$.
\end{prop}
\begin{proof}
The augmentation ideal of $k[\G/\cN]\cong k[\mu_{(p^{r-1})}]=k[T]/(T^{p^{r-1}}-1)$ is a principal ideal.
By \cite[4.2.3]{farnsteiner2009group} and \cite[4.3]{farnsteiner2009group} the induced $\G$-module $k\G\otimes_{k\cN}Z(a)^{g}$ is
a brick whose restriction to $\cN$ is isomorphic to $g.W(p^r+a)$.
By \cite[4.1.2]{farnsteiner2009group}, the module $Z(a)^{g}$ is quasi-simple and belongs to a homogeneous
tube $\Theta$ of the stable Auslander-Reiten quiver $\Gamma_s(\cN)$. Additionally, \cite[4.1.2]{farnsteiner2009group} yields that
$g.W(lp+a)$ is the $\cN$-module of quasi-length $l$ in $\Theta$.
The assertion now follows by applying \ref{filtration of induced quasi simple modules}.
\end{proof}

\begin{rmk}
	The above result can also be applied if $g\in\{1,w_0\}$. One only
	has to use another torus $\hat{T}$ such that the induction of $Z(a)$ respectively $Z(a)^{w_0}$ to
	$SL(2)_1\hat{T}_r$ is indecomposable.
\end{rmk}

We are now able to use the filtration of induced modules 
to realize the $SL(2)_1T_r$-modules which belong to homogeneous tubes.
Moreover, for a subgroup scheme $H$ of $N_{SL(2)}(T)\cap B^g$
we are able to extend these modules to $SL(2)_1T_rH$,
which later will be of interest for the classification of modules for domestic group schemes.\\
Denote by $B^g$ the subgroup of $SL(2)$ which is obtained by conjugating all elements of $B$ by $g$.
Then $Z(a)^{g}\cong g.W(p+a)$  is stable under the action of $B^g$, so that $Z(a)^{g}$ is an $SL(2)_1B^g$-module.
\begin{thm}\label{extending x(i,g,l)}
Let $g\in C\setminus\{1,w_0\}$, $0\leq a\leq p-2$ and $H$ be a subgroup scheme of $N_{SL(2)}(T)\cap B^g$.
For $n\geq 1$ let $\cH_{(n)}:=SL(2)_1T_nH$ and $\cN:=SL(2)_1$.
Let $r,s\geq 1$, $Y:=k{\cH_{(r)}}\otimes_{k{\cH_{(1)}}}Z(a)^{g}$ and denote the filtration by $\cH_{(r)}$-modules of the induced module $N:=k\cH_{(r+s)}\otimes_{k\cH_{(r)}}Y$ by
\[k\cH_{(r)}\otimes_{k{\cH_{(1)}}}Z(a)^{g}\cong N_{0}\subseteq N_{1}\subseteq\dots\subseteq N_{p^{s}-1}=N.\]
Then $\res_\cN^{\cH_{(r)}} N_{l-1}\cong g.W(lp^r+a)$ for all $1\leq l\leq p^s$.
\end{thm}
\begin{proof}
Set $\G:=SL(2)_1T_{r+s}$.
As $\cH_{(r+s)}/\cH_{(1)}\cong\mu_{(p^{r+s-1})}\cong \G/\cN$ we obtain $k[\cH_{(r+s)}]^{\cH_{(1)}}\cong k[\mu_{(p^{r+s-1})}]\cong k[\G]^\cN$.
Denote the filtration by $\cH_{(1)}$-modules of the induced module $N\cong k\cH_{(r+s)}\otimes_{k{\cH_{(1)}}}Z(a)^{g}$ by
\[Z(a)^{g}\cong M_{0}\subseteq M_{1}\subseteq\dots\subseteq M_{p^{r+s-1}-1}=N.\]
The $\cH_{(r+s)}$-module $N$ is over $\G$ isomorphic to $k\G\otimes_{k\cN}Z(a)^{g}$. These modules are also isomorphic over 
$k[\mu_{(p^{r+s-1})}]$ with respect to the action defined in section \ref{sec: filtrations of induced modules}.
Applying \ref{filtration of induced} yields that the restriction of the modules $M_i$ to $\cN$ is the
filtration by $\cN$-modules of the induced module $k\G\otimes_{k\cN}Z(a)^{g}$.
Let $J$ be the kernel of the canonical projection $k[\cH_{(r+s)}]\rightarrow k[\cH_{(1)}]$. Then the ideal $J^{p^{r-1}}$ is the kernel of the canonical projection
$k[\cH_{(r+s)}]\rightarrow k[\cH_{(r)}]$. By \ref{filtration of induced}, we get the equality 
$M_{p^{r-1}l-1}=\res_{\cH_{(1)}}^{\cH_{(r)}}N_{l-1}$ for all $1\leq l\leq p^s$. \\
By \ref{X(i,g,l)}, there is for any $l\geq 1$ a unique $SL(2)_1T_r$-module $X(i,g,l)$ which is isomorphic to $g.W(lp^r+a)$
 over $\cN$. Thanks to \ref{new realization of sl(2) modules}, this module is isomorphic to $\res_\cN^{\cH_{(1)}} M_{p^{r-1}l-1}=\res_\cN^{\cH_{(r)}}N_{l-1}$.
\end{proof}

%% file: quot_sup_var.tex
\section{Actions on rank varieties and their stabilizers}\label{sec:action on rank varieties}
Let $k$ be a field of characteristic $p>0$, $\mathcal{G}$ a finite group scheme and $\cg:=\Lie(\G)$ its Lie algebra.
The nullcone $V_\cg$ is a cone, so that we can consider the projective variety $\P(V_\cg)$. 
There is an action of the group-like elements of $k\G$ on its primitive elements and therefore we obtain an action of $\G(k)$ on $\cg$. Moreover, this action induces an action of $\G(k)$ on $\P(V_\cg)$. \\
 For a variety $X$ 
and a point $x\in X$ the tangent space
$T_{x,X}$ is defined as the dual space $(\mathfrak{m}_{x,X}/\mathfrak{m}_{x,X}^2)^*$ where $\mathfrak{m}_{x,X}$ denotes the
maximal ideal of the local ring $\mathcal{O}_{x,X}$ of $X$ at $x$. A point $x\in X$ is called simple, if $\mathcal{O}_{x,X}$ is
a regular local ring.
The following result can be found in \cite[Lemma 4]{popov2013jordan} for $\Char k=0$, 
but the proof can easily be modified such that
it applies to finite groups whose order is relatively prime to the characteristic of the field.
\begin{lem}\cite[Lemma 4]{popov2013jordan}\label{tangent action faithful}
Let $X$ be an irreducible variety and $G$ be a finite group with $p\nmid\vert G\vert$ which acts faithfully on $X$. 
Let $x\in X$ be a fixed point of $G$. Then the induced action of $G$ on $T_{x,X}$ is faithful.
\end{lem}
\begin{rmk}
	The result can also be generalized to finite linearly reductive group schemes acting on $X$.
	Consequently there are also generalizations of the following results to this situation.
\end{rmk}
\begin{lem}\label{stabilizer is subgroup of GL_r}
Let $\G$ be a finite group scheme with Lie algebra $\cg:=\Lie(\G)$ such that the variety $\P(V_{\cg})$ is irreducible.
Assume that the order of $G:=\G(k)$ is relatively prime to $p$ and that $G$ acts faithfully on $\P(V_{\cg})$.
Moreover, let $r:=\dim \P(V_{\cg})$ and $x\in \P(V_{\cg})$ be a simple point. Then there is an injective homomorphism $G_x\rightarrow GL_r(k)$.
\end{lem}
\begin{proof}
Since $x$ is a fixed point of $G_x$ and $x$ is a simple point, the action of $G_x$ on $T_{x,\P(V_{\cg})}$ is faithful,
by Lemma \ref{tangent action faithful}.
As the point $x$ is simple, we have $r=\dim_kT_{x,\P(V_{\cg})}$. So, there is an injective homomorphism
$G_x\rightarrow GL(T_{x,\P(V_{\cg})})\cong GL_r(k)$.
\end{proof}

Any finite group scheme decomposes into a semi-direct product $\G^0\rtimes\G_{red}$ with an infinitesimal normal subgroup scheme $\G^0$
and a reduced group scheme $\G_{red}$. The group algebra $k\G$ is isomorphic to the skew group algebra $(k\G^0)*G$ where $G=\G(k)$.
We denote by $G_M:=\{g\in G\:\vert\:M^g\cong M\}$ the stabilizer (or inertia group) of a $\G^0$-module $M$. 
The rank variety of the twisted module $M^g$ can be computed as $\P(V_{\cg}(M^g))=g.\P(V_{\cg}(M))$.
If $\P(V_{\cg}(M))=\{x\}$, then it is easy to see that $G_M\subseteq G_x$.
\begin{rmk}\label{stabilizers equal}
Let $\G\subseteq SL(2)$ with $\G^0\cong SL(2)_1$ and $M$ be a $\G^0$-module which belongs to a homogeneous tube $\Theta$.  Then there are $g\in SL(2,k)$ and $d\in\N$ with $M\cong g.W(d)$ and $\P(V_{\sl2})(g.W(d))=\{g.[e]\}$.
Let $h\in G_{g.[e]}$. Then $g^{-1}hg.[e]=[e]$ and hence $g^{-1}hg$ is an element of the standard Borel subgroup of upper triangular matrices $B$. From this follows that $hg.W(d)=g.W(d)$ and thus
$h\in G_{g.W(d)}$. Therefore we obtain
\[G_{g.W(d)}=G_{g.[e]}.\]
\end{rmk}
\begin{cor}\label{inertia group cyclic for tubes}
	Let $\mathcal{G}$ be a finite group scheme with Lie algebra $\cg=\Lie(\G)$ such that $\P(V_{\cg})$ is smooth and irreducible.
	Assume that $\G/\G_1$ is linearly reductive and that $G=\G(k)$ acts faithfully on $\P(V_{\cg})$.
	Let $r:=\dim\P(V_{\cg})$ and $M$ be an indecomposable $\mathcal{G}^0$-module of complexity $1$.
	Then there is an injective homomorphism $G_M\rightarrow GL_r(k)$.
	If additionally $r=1$ (for example when $k\mathcal{G}^0$ is tame), then $G_M$ is a cyclic group.
\end{cor}
\begin{proof}
	By By \cite[3.15]{doi1989hopf}, 
	the extension $k\G^0:k\G_1$ is separable, so that
	there is an indecomposable direct summand $N$ of $\res_{\G_1}^{\G^0} M$ such that $M$ is a direct
	summand of $\ind^{\G^0}_{\G_1} N$. Thanks to \cite[2.1.2]{farnsteiner2009group} the module $\res_{\G_1}^{\G^0} M$
	is indecomposable and therefore $N=\res_{\G_1}^{\G^0} M$.
	General properties of the complexity (\cite[II.2]{shanghai}) yield
	\[1=\cx_{k\G^0}(M)\leq\cx_{k\G^0}(\ind_{\G_1}^{\G^0} N)\leq\cx_{k\G_1}(N)\leq\cx_{k\G_1}(M)\leq\cx_{k\G^0}(M)=1.\]
	Hence $N$ has also complexity $1$ and the variety $\P(V_{\sl2}(M))$ has dimension $0$.
	By \cite[2.2]{friedlander1986support}, the indecomposability of $M$ yields that the variety $\P(V_{\sl2}(M))$ consists of only one point $x$. Thus $G_M\subseteq G_x$.
	The first assertion now follows from \ref{stabilizer is subgroup of GL_r}. 
	If $r=\dim\P(V_{\cg})=1$,  then $G_M$ is
	isomorphic to a finite subgroup of $k^\times$ and therefore cyclic.
\end{proof}
\begin{example}
	Let $\cH$ be a finite reduced linearly reductive subgroup of $GL_n$. Then $\cH$ acts naturally 
	on the $n$-fold product $\mathbb{G}_{a(1)}^n$ of the first Frobenius kernel of the additive group
	$\mathbb{G}_a$. Hence we can form the semi-direct product
	$\G:=\mathbb{G}_{a(1)}^n\rtimes\cH$. Then $\G^0=\mathbb{G}_{a(1)}^n$, $\G_{red}=\cH$
	and $\cg:=\Lie(\G^0)$ is an $n$-dimensional abelian restricted Lie algebra with trivial $p$-map.
	The nullcone can be computed as $\P(V_{\cg})\cong\P(\cg)\cong\P^{n-1}$
	and $G:=\G(k)=\cH(k)$ acts faithfully on this variety. Hence we can apply \ref{inertia group cyclic for tubes}, so that
	the stabilizer $G_M$ of every periodic $\G^0$-module $M$ is isomorphic to a finite subgroup of $GL_{n-1}(k)$.
\end{example}

%% file: decomp_ind_mod.tex
\section{Decomposition of induced modules}\label{sec:decomp induced module}
Let $H$ be a Hopf algebra and $A$ be an $H$-comodule $k$-algebra (i.e. $A$ is a $k$-algebra and an $H$-comodule such that the comodule map $\rho_A:A\rightarrow A\otimes H$
is an algebra homomorphism). Denote by $B:=A^{coH}$ the coinvariants of $H$. Then $B\subseteq A$ is called $H$-extension.
An $H$-extension is called $H$-Galois if the map $\beta:A\otimes_B A\rightarrow A\otimes_k H$, $a\otimes b \mapsto a\rho_A(b)$ is bijective
(\cite[\S 1]{schneider1990representation}).
If $\cN$ is a normal subgroup scheme of a finite group scheme $\G$, then the extension $k\cN\subseteq k\G$ is $k(\G/\cN)$-Galois.
Let $B\subseteq A $ be an $H$-Galois extension and $M$ be an $A$-module. Then $\End_B(M)$ is an $H$-module algebra via
\[(h.f)(m)=\sum_{i=1}^{n}a_if(b_im)\]
for $h\in H$, $f\in\End_B(M)$, $m\in M$ and $\sum_{i=1}^{n}a_i\otimes b_i=\beta^{-1}(1\otimes \eta(h))\in A\otimes_B A$, where $\eta$
is the antipode of $H$.
The analogue of \cite[2.3]{van1995h} for left modules shows that the endomorphism algebra $\End_A(A\otimes_B M)$ of the induced module $A\otimes_B M$
is isomorphic to the smash product $\End_B(M)\# H$. \\
We want to use this result to describe the decomposition of certain induced modules over group schemes.
The modules we want to consider satisfy a stronger stability condition then those considered in \cite{voigt1977induzierte},
\cite{schneider1990representation} and \cite{van1995h}. \\
Let $\cM$ be a multiplicative group scheme, i.e. the coordinate ring of $\cM$ is isomorphic
to the group algebra $kX(\cM)$ of its character group. 
Since $\cM$ is multiplicative, we obtain for any $\cM$-module $M$ a weight space decomposition
$M=\bigoplus_{\lambda\in X(\cM)}M_\lambda$ with
\[M_\lambda:=\{m\in M\:\vert\:hm=\lambda(h)m\;\text{for all}\;h\in k\cM\}.\]
Let $\G$ be a finite group scheme and
$\cN\subseteq\G^0$ a normal subgroup scheme of $\G$ such that $\G^0/\cN$ is linearly reductive.
By Nagata's Theorem \cite[IV,\S 3,3.6]{demazure1970groupes}, an infinitesimal group scheme is linearly reductive if
and only if it is multiplicative. \\
Let $M$ be a $\G^0$-module. For any $\lambda\in X(\G^0/\cN)$ we obtain a $\G^0$-module
$M\otimes_kk_\lambda$, the tensor product of $M$ with the one-dimensional $\G^0/\cN$-module defined
by $\lambda$. This defines an action of $X(\G^0/\cN)$ (here we identify $X(\G^0/\cN)$ with a subgroup of 
$X(\G^0)$ via the canonical inclusion $X(\G^0/\cN)\hookrightarrow X(\G^0)$) on the isomorphism classes of $\G^0$-modules
and we define the stabilizer of a $\G^0$-module as
$X(\G^0/\cN)_M:=\{\lambda\in X(\G^0/\cN)\:\vert\:M\otimes_kk_\lambda\cong M\}$. A $\G^0$-module
$M$ is called $\G^0/\cN$-regular if its stabilizer is trivial.

\begin{prop}\label{decomp of induced stable modules}
Let $\G$ be a finite group scheme over an algebraically closed field $k$ and $\cN$ be an infinitesimal normal subgroup scheme such that $\G^0/\cN$ is multiplicative. 
Let $k(\G/\cN)=\bigoplus_{i=1}^nP_i$ be the decomposition into projective indecomposable $\G/\cN$-modules. 
Let $M$ be a $\G$-module such that $\res_{\G^0}^\G M$ is indecomposable and $M$ is $\G^0/\cN$-regular.
Then $k\G\otimes_{k\cN}M$ is isomorphic
to the direct sum $\bigoplus_{i=1}^nM\otimes_k P_i$ of indecomposable $\G$-modules and 
$M\otimes_k P_i\cong M\otimes_k P_j$ as $\G$-modules if and only if $P_i\cong P_j$ as $\G/\cN$-modules.
\end{prop}
\begin{proof}
Let $G=\G(k)$. Then $k\G^0\subseteq k\G$ is a $kG$-Galois extension and $k\cN\subseteq k\G^0$
is a $k(\G^0/\cN)$-Galois extension.
We set $E:=\End_\G(k\G\otimes_{k\cN}M)$ and $E':=\End_\cN(M)$.
As $M$ is a $\G$-module, we can apply the above twice to get isomorphisms 
\[E\cong \End_{\G^0}(k\G^0\otimes_{k\cN}M)\#kG\cong(E'\# k(\G^0/\cN))\#kG.\]
Denote by $\pi:k\G^0\rightarrow k(\G^0/\cN)$ the canonical projection. Then the comodule map $\rho:k\G^0\rightarrow k\G^0\otimes k(\G^0/\cN)$
is given by 
$\rho(a)=\sum_{(a)}a_{(1)}\otimes\pi(a_{(2)})$. This yields 
\[\beta(\sum_{(a)}\eta(a_{(1)})\otimes a_{(2)})=1\otimes\pi(a).\]
Therefore the $k(\G^0/\cN)$-action on $E'$ is given by 
\[(h.f)(m)=\sum_{(a)}a_{(1)}f(\eta(a_{(2)})m)\] for $h\in k(\G^0/\cN)$, $f\in E'$,
$m\in M$ and $a\in k\G^0$ with $\pi(a)=h$. Hence we obtain $(E')^{\G^0/\cN}=\End_{\G^0}(M)$. \\
Since $\G^0/\cN$ is multiplicative, the space $E'$ affords a decomposition into weight spaces  $E'=\bigoplus_{\lambda\in X(\G^0/\cN)}E'_\lambda$
with $E'_0=(E')^{\G^0/\cN}=\End_{\G^0}(M)$. As $M$ is regular we have: (c.f. for example the proof of \cite[3.1.4]{farnsteiner2009group})
\[\rad(E')=\rad(\End_{\G^0}(M))\oplus\bigoplus_{\substack{\lambda\in X(\G^0/\cN)\\\lambda\neq 0}}E'_\lambda.\]
Hence the Jacobson radical $\rad(E')$ is stable under the action of $\G^0/\cN$. By \cite[3.2(2)]{zhang1998baer}, we get
\[\rad(E')\#k(\G^0/\cN)\subseteq\rad(E'\#k(\G^0/\cN)).\]
Moreover, by \cite[4.3]{cohen1984group}, we have \[\rad(E'\#k(\G^0/\cN))\#kG\subseteq \rad((E'\#k(\G^0/\cN))\#kG).\]
Therefore, $(\rad(E')\#k(\G^0/\cN))\#kG$ is nilpotent.
By \cite[2.1.2]{farnsteiner2009group}, the indecomposability of $\res_{\G^0}^\G M$ yields that $\res_\cN^\G M$ is
indecomposable and therefore $E'$ is local. As $k$ is algebraically closed, we obtain isomorphisms
\begin{align*}
&(E'\#k(\G^0/\cN))\#kG)/((\rad(E')\#k(\G^0/\cN))\#kG)\\
\cong& (k\#k(\G^0/\cN))\#kG\cong k(\G^0/\cN)\#kG\\
\cong& k(\G/\cN).
\end{align*}
Now let $p:E\rightarrow k(\G/\cN)$ be the surjection given by the above isomorphisms. 
Then the map $\psi:k\G\otimes_{k\cN}M\otimes_{E^{op}}E^{op}\rightarrow M\otimes_k k(\G/\cN)$ with
$\psi(a\otimes m\otimes \phi)=\sum_{(a)}a_{(1)}m\otimes \pi(a_{(2)})p(\phi)$ is an isomorphism of $\G$-modules:\\
The $\G$-linearity follows directly from the definition of the operation on $M\otimes_k k(\G/\cN)$. Let $m\in M$
and $y\in k\G$. Then
\begin{align*}
&\psi(\sum_{(y)}y_{(1)}\otimes \eta(y_{(2)})m\otimes 1)
=\sum_{(y)}y_{(1)}\eta(y_{(2)})m\otimes \pi(y_{(3)})\\
=&\sum_{(y)}\epsilon(y_{(1)})m\otimes \pi(y_{(2)})
=m\otimes\pi(y),
\end{align*}
so that $\psi$ is surjective and therefore bijective for dimension reasons.\\
Hence we have a decomposition  $k\G\otimes_{k\cN}M\cong\bigoplus_{i=1}^nM\otimes_k P_i$ of $\G$-modules.
Now let $e_1,\dots,e_n\in k(\G/\cN)$ be primitive idempotents with $P_i=k(\G/\cN)e_i$ and set $Q_i:=e_ik(\G/\cN)$.
Then $k(\G/\cN)^{op}=\bigoplus_{i=1}^nQ_i$ is a decomposition into indecomposable left ideals and by 
\cite[4.5.11]{karpilovsky1987algebraic} this
decomposition lifts to a decomposition $E^{op}=\bigoplus_{i=1}^nI_i$ of indecomposable left ideals of $E^{op}$ such that $p(I_i)=Q_i$.
Let $f_1,\dots,f_n\in E^{op}$ be primitive idempotents such that $p(f_i)=e_i$.
Then \[\psi(k\G\otimes_{k\cN}M\otimes_{E^{op}}I_i)=M\otimes_k k(\G/\cN)p(f_i)=M\otimes_k k(\G/\cN)e_i=M\otimes_k P_i.\]
By \cite[4.5.12]{karpilovsky1987algebraic} the $\G$-module $k\G\otimes_{k\cN}M\otimes_{E^{op}}I_i$ is indecomposable 
so that $M\otimes_k P_i$ is also indecomposable.
Moreover, we have $M\otimes_k P_i\cong M\otimes_k P_j$ as a $\G$-module if and only if
$k\G\otimes_{k\cN}M\otimes_{E^{op}}I_i$ is isomorphic to $k\G\otimes_{k\cN}M\otimes_{E^{op}}I_j$ as a $\G$-module and by
\cite[4.5.12]{karpilovsky1987algebraic}these are isomorphic if and only if $I_i\cong I_j$ as left ideals of $E^{op}$.
By \cite[4.5.11]{karpilovsky1987algebraic} we have $I_i\cong I_j$ as left ideals of $E^{op}$ if and only if $Q_i\cong Q_j$ as
left ideals of $k(\G/\cN)^{op}$ and these are isomorphic if and only if $P_i\cong P_j$ as $\G/\cN$-modules.
%
%
\end{proof}

\begin{example}
	There is a somewhat weaker result for non-regular $M$. Let $\G$ be infinitesimal and $M\cong M\otimes_kk_\lambda$ for all $\lambda\in X(\G/\cN)$.
	By \cite[2.1.5]{farnsteiner2009group}, the ring $\End_\cN(M)$ is isomorphic to a crossed product
	$\End_\G(M)\#_\sigma kX(\G/\cN)$. Since $M$ is a $\G$-module, we get as above an isomorphism
	$\End_\G(k\G\otimes_{k\cN}M)\cong\End_\cN(M)\# k(\G/\cN)$. As $(k(\G/\cN))^*\cong kX(\G/\cN)$,
	there is, due to \cite[9.4.17]{susan1993hopf}, an isomorphism
	$E:=\End_\G(k\G\otimes_{k\cN}M)\cong\End_\G(M)\otimes M_n(k)$.
	Hence $\rad(E)\cong\rad(\End_\G(M))\otimes_kM_n(k)$ and  $E/\rad(E)\cong M_n(k)$. Consequently \[k\G\otimes_{k\cN}M\cong M^n\cong\bigoplus_{\lambda\in X(\G/\cN)}M\otimes_kk_\lambda.\] 
	
\end{example}

%% file: ar_group_graded.tex
\section{Auslander-Reiten quiver of group graded algebras}\label{sec:AR-quiver group graded}
Let $G$ be a finite group and $A=\bigoplus_{g\in G}A_g$ a self-injective strongly $G$-graded finite-dimensional $k$-algebra,
i.e. we have $A_gA_h=A_{gh}$ for all $g,h\in G$. By \cite[Corollary 2.10]{nastasescu1983strongly}, the self-injectivity
of $A$ is equivalent to the self-injectivity of $A_1$.
For a subgroup $H\leq G$ we denote by $A_H$ the strongly $H$-graded subalgebra $\bigoplus_{g\in H}A_g$
of $A$. 
For subgroups $H\leq U\leq G$ we abbreviate the induction and restriction functors by
$\ind_H^U:=\ind_{A_H}^{A_U}$ and $\res_H^U:=\res_{A_H}^{A_U}$. For $H=\{1\}$ we will
write $\ind_1^U$ and $\res_1^U$. \\
A morphism $\sigma:(\Gamma,\tau_\Gamma)\rightarrow(\Lambda,\tau_\Lambda)$ of stable translation quivers is a morphism of quivers which commutes with the
translation $\sigma\circ\tau_\Gamma=\tau_\Lambda\circ\sigma$. For a stable translation quiver $(\Gamma,\tau_\Gamma)$ we will denote by
$\Aut(\Gamma)=\Aut(\Gamma,\tau_\Gamma)$ its automorphism group.

\begin{rmk} 
	If $\Gamma=\Z[A_\infty]/(\tau^n)$ is an exceptional tube of rank $n$, then the group $\aut(\Gamma)=\langle\tau_\Gamma\rangle$ has order $n$. 
\end{rmk}
The group $G$ acts on the module category $\mod A_1$  via equivalences of categories
$\mod A_1\rightarrow\mod A_1, M\mapsto M^g$ for $g\in G$. Since these equivalences commute with the Auslander-Reiten translation of $\Gamma_s(A_1)$,
each $g\in G$ induces an automorphism $t_g$ of the quiver $\Gamma_s(A_1)$. As $t_g$ permutes the components of $\Gamma_s(A_1)$, we can conclude
that $G$ acts on the set of components of $\Gamma_s(A_1)$. For a component $\Theta$ we write $\Theta^g=t_g(\Theta)$ and let 
$G_\Theta=\{g\in G\:\vert\: \Theta^g=\Theta\}$ be the stabilizer of $\Theta$.
If $M$ is an $A_1$-module which belongs to the component $\Theta$, then $G_M\subseteq G_\Theta$.

\begin{lem}\label{stabilizer of component equals stabilizer of module}
Let $\Theta$ be a component of $\Gamma_s(A_1)$ with finite automorphism group $\aut(\Theta)$ such that $\vert G_\Theta\vert$ and $\vert\aut(\Theta)\vert$ 
are relatively prime. Let $M$ be an $A_1$-module which belongs to $\Theta$. Then $G_M=G_\Theta$.
\end{lem}
\begin{proof}
The action of $G_\Theta$ on $\Theta$ induces a homomorphism $\psi:G_\Theta\rightarrow\aut(\Theta)$ of groups.
The kernel of this homomorphism is given by $\ker\psi=\bigcap_{N\in\Theta}G_N$. 
Since $\vert G_\Theta\vert$ and $\vert\aut(\Theta)\vert$ have no common divisor, the homomorphism $\psi$ is trivial.
Hence $G_\Theta=\bigcap_{N\in\Theta}G_N$ and thus $G_M=G_\Theta$.
\end{proof}

Now let $N$ be an indecomposable non-projective $A_1$-module and $\Xi$ the corresponding component in $\Gamma_s(A_1)$. Assume there is an
indecomposable non-projective direct summand $M$ of $\ind_1^GN$ and let $\Theta$ be the corresponding component in $\Gamma_s(A)$.
Since $G_N$ is contained in $G_\Xi$, for every indecomposable direct summand $V$ of $\res^G_{G_\Xi}\ind_1^GN$ the $A$-module $\ind_{G_\Xi}^G V$
is indecomposable (\cite[4.5.2]{karpilovsky1987algebraic}).
Hence there is an indecomposable direct summand $U$ of $\res_{G_\Xi}^GM$ such that $\ind_{G_\Xi}^GU=M$.
Denote by $\Psi$ the component of $U$ in $\Gamma_s(A_{G_\Xi})$.

\begin{lem}\label{multiplicity of tubes for induction with cyclic group}
Let $k$ be an algebraically closed field.
Suppose that all $A_1$-modules which belong to $\Xi$ are $G_\Xi$-stable and that $G_\Xi$ is a cyclic group such that $\Char k\nmid \vert G_\Xi\vert$.
Then we have the following:
\begin{enumerate}[(a)]
	\item $\res_1^{G_\Xi}:\Psi\rightarrow\Xi$, $[X]\rightarrow [\res_1^{G_\Xi}X]$ is a morphism of stable translation quivers,
	\item for all $[Y]\in\Xi$ we have $\vert(\res_1^{G_\Xi})^{-1}([Y])\vert\leq\vert G_\Xi\vert$,
	\item if $\Xi$ and $\Psi$ have tree class $A_\infty$, then $\res_1^{G_\Xi}:\Psi\rightarrow\Xi$ preserves the quasi-length, and
	\item if $\Xi$ and $\Psi$ are tubes of finite rank $n$ and $m$, then $m\leq\vert G_\Xi\vert n$.
\end{enumerate}
\end{lem}
\begin{proof}
We first show that under our assumptions the restriction of every $A_{G_\Xi}$-module in $\Psi$ is an indecomposable $A_1$-module which belongs to $\Xi$. 
Let $V$ be an indecomposable $A_{G_\Xi}$-module in $\Psi$ and let $\res_1^{G_\Xi}V=\bigoplus_{i=1}^nU_i$ be its decomposition into
indecomposable $A_1$-modules. By \cite[4.5.8]{marcus1999representation}, all these modules belong to $\bigcup_{g\in G_\Xi}\Xi^g=\Xi$
and therefore are $G_\Xi$-stable. Let $r=\vert G_\Xi\vert$. Since $G_\Xi$ is cyclic, $\Char k\nmid \vert G_\Xi\vert$ 
and $k$ is algebraically closed, we get due to \cite[4.5.15,4.5.17]{karpilovsky1987algebraic}
a decomposition $\ind_1^{G_\Xi}U_i=\bigoplus_{j=1}^rW_{i,j}$ into indecomposable $A_{G_\Xi}$-modules of dimension $\dim_kW_{i,j}=\dim_kU_i$  for all $i\in\{1,\dots,n\}$.
In particular, the restriction $\res_1^{G_\Xi}W_{i,j}$ is isomorphic to $U_i$.
Thanks to \cite[2.3.4]{karpilovsky1987algebraic} and \cite[2.7.24]{karpilovsky1987algebraic}, the
module $V$ is a direct summand of $\ind_1^{G_\Xi}\res_1^{G_\Xi}V=\bigoplus_{i=1}^n\bigoplus_{j=1}^rW_{i,j}$ and therefore
isomorphic to one of the $W_{i,j}$. In particular, the module $\res_1^{G_\Xi}V\cong\res_1^{G_\Xi}W_{i,j}\cong U_i$ is indecomposable.
\begin{enumerate}[(a)]
	\item 
	Let $X\rightarrow Y$ be an arrow in $\Psi$. Then there is an almost split sequence of $A_{G_\Xi}$-modules
	\[
	\mathcal{E}:\:\ses{\tau_{\Psi}(Y)}{E}{Y}{}{}
	\]
	such that $X$ is a direct summand of $E$ and the indecomposable $A_1$-modules $\res_1^{G_\Xi}X$ and $\res_1^{G_\Xi}Y$ belong to $\Xi$.
	By \cite[Theorem 6]{solberg1989strongly}, the sequence $\res_1^{G_\Xi}\mathcal{E}$ is a direct sum of almost split sequences. Since $\res_1^{G_\Xi}Y$ and $\res_1^{G_\Xi}\tau_{\Psi}(Y)$ are
	indecomposable, the sequence $\res_1^{G_\Xi}\mathcal{E}$ is almost split. In particular, $\res_1^{G_\Xi}\tau_{\Psi}(Y)\cong\tau_{\Xi}(\res_1^{G_\Xi}Y)$.
	Moreover, this gives us an arrow $\res_1^{G_\Xi}X\rightarrow\res_1^{G_\Xi}Y$. Therefore
	$\res_1^{G_\Xi}:\Psi\rightarrow\Xi$, $[X]\mapsto [\res_1^{G_\Xi}X]$ is a morphism of stable translation quivers.
	\item
	Let $[Y]\in\Xi$ and $[X]\in\Psi$ with $\res_1^{G_\Xi}([X])=[Y]$. As before, we have a decomposition 
	$\ind_1^{G_\Xi}Y=\bigoplus_{i=1}^rY_i$ into indecomposable $A_{G_\Xi}$-modules and
	$X$ is a direct summand of $\ind_1^{G_\Xi}\res_1^{G_\Xi}X=\ind_1^{\G_\Xi}Y=\bigoplus_{i=1}^rY_i$.
	Therefore the number of preimages of $[Y]$ is bounded by $r$.
	\item 
 We only need to show that $\res_1^{G_\Xi}:\Psi\rightarrow\Xi$
 sends quasi-simple modules to quasi-simple modules.
 Let $Y$ be a quasi-simple module in $\Psi$ and let
 \[\ses{\tau_\Psi(Y)}{E}{Y}{}{}\]
 be the almost split sequence ending in $Y$.
 As shown in (a), the sequence $\res_1^{G_\Xi}\mathcal{E}$ is almost split.
 As $Y$ is quasi-simple, the module $E$ is the direct sum $X\oplus P$ of an indecomposable module $X$ and a projective
 module $P$. Since $X$ belongs to $\Psi$, the module
 $\res_1^{G_\Xi}X$ is indecomposable, so that $\res_1^{G_\Xi}E=\res_1^{G_\Xi}X\oplus\res_1^{G_\Xi}P$
 is the direct sum of an indecomposable and a projective module.
 Hence $\res_1^{\G_\Xi}Y$ is quasi-simple.
 \item   Let $Y_1,\dots,Y_n$ be the quasi-simple modules in $\Xi$. 
As $\res_1^{\G_\Xi}$ preserves the quasi-length, every module belonging to $(\res_1^{\G_\Xi})^{-1}([Y_i])$ is quasi-simple.
  Applying (b) yields that $\Psi$ has, up to isomorphism, at most $rn$ quasi-simple modules.
\end{enumerate}
\end{proof}

%% file: mod_dom_group_scheme.tex
\section{Modules of domestic group schemes}\label{sec:modules domestic groups}
Let $k$ be an algebraically closed field of characteristic $p>2$. In this section we will prove the main results 
for the classification of the modules for a certain class of domestic group schemes, the so-called amalgamated polyhedral group schemes.
Such a group scheme is the image of $SL(2)_1\tilde{\G}$ under the canonical projection $SL(2)\rightarrow PSL(2)$ where $\tilde{\G}$
denotes a binary polyhedral group scheme, i.e. a finite linearly reductive subgroup scheme of $SL(2)$.\\
For any domestic group scheme $\G$ the factor group $\G/\G_{lr}$, by the largest linearly reductive normal subgroup scheme $\G_{lr}$ of $\G$,
is isomorphic to an amalgamated polyhedral group scheme.
The principal blocks of these group schemes are isomorphic and all non-simple blocks of the group algebra $k\G$ 
are Morita-equivalent to this block.\\
At the end of this section we will, by way of example, classify the modules in the most complicated case, the amalgamated dihedral group schemes.
\begin{example}
	For $m\in\N$ consider the subgroup $T_{(m)}\subseteq T$ given by
	\[T_{(m)}(R):=\{\left(\begin{smallmatrix}
	x & 0 \\ 0 & x^{-1}
	\end{smallmatrix}\right)\;\vert\; x\in\mu_{(m)}(R)\},\] for any commutative $k$-algebra $R$.
	Let $\pi:SL(2)\rightarrow PSL(2)$ be the canonical projection. Then the image faisceau
	$P\cS\cC_{(m)}:=\pi(SL(2)_1N_{(2m)})$ of $\pi$ restricted to $SL(2)_1N_{(2m)}$
	is called amalgamated cyclic group scheme. \\
	For $m\geq 2$ let $N_{(m)}:=T_{(m)}\cH_4$ be a binary dihedral group scheme as defined in 
	\cite[Section 3]{farnsteiner2006polyhedral}. By definition,  
	$\cH_4$ is a reduced subgroup scheme of $N_{SL(2)}(T)$ with $\cH_4(k)=<w_0>$ and
	there is $h_4\in GL(2)(k)$ with $\cH_4=h_4T_{(4)}h_4^{-1}$.
	Then the amalgamated dihedral group scheme is given by
	$P\cS\cQ_{(m)}:=\pi(SL(2)_1N_{(2m)})$.
\end{example}
For any amalgamated polyhedral group scheme $\G$ there is an $r\geq 1$ such that $\G^0$ is isomorphic to
$SL(2)_1T_r$.
Recall that for $l\in\N$, $i\in\{0,\dots,p-2\}$ and $g\in SL(2)\setminus(B\cup w_0B)$
there is a unique $SL(2)_1T_r$-module $X(i,g,l)$ such that $\res_{SL(2)_1}X(i,g,l)\cong g.W(lp^r+i)$.
For $g\in B\cup w_0B$, the $SL(2)_1$-module $g.W(lp+i)$ is stable under the action of $SL(2)_1T_r$.
By abuse of notation we write in this situation $X(i,g,l)=g.W(lp+i)$.\\

Let $\pi:SL(2)\rightarrow PSL(2)$ be the canonical projection. The kernel of $\pi$ is the reduced group scheme generated by
$\left(\begin{smallmatrix}
-1 & 0 \\
0  & -1
\end{smallmatrix}\right)$.
Let $H:=N_{SL(2)}(T)\cap B^g$ and $\G:=\pi(SL(2)_1T_rH)$. For all $g\in SL(2)(k)$ the module $X(i,g,l)$ can be viewed as an $SL(2)_1T_rH$-module as in \ref{extending x(i,g,l)} and 
$\left(\begin{smallmatrix}
-1 & 0 \\
0  & -1
\end{smallmatrix}\right)$
acts via $(-1)^i$. Therefore, for even $i$ the action factors through $\pi$, so that $X(i,g,l)$ is a $\G$-module. 
For odd $i$ we need to twist the action by a special character
$\gamma:SL(2)_1T_rH\rightarrow \mu_{(1)}$.
To define this character consider the homomorphism \[\phi:SL(2)_1T_rH\rightarrow SL(2)_1T_r^{g^{-1}}H^{g^{-1}}\;,\;x\mapsto x^{g^{-1}}\]
of group schemes and the character \[\tilde{\gamma}:SL(2)_1T_r^{g^{-1}}H^{g^{-1}}\rightarrow\mu_{(1)}\;,\; 
\left(\begin{smallmatrix}
a & b \\
c  & d
\end{smallmatrix}\right)
\mapsto d^{p^r}\] (this is a character as $H^{g^{-1}}\subseteq B$ and as it is trivial on $SL(2)_1T_r^{g^{-1}}$).
Then $\gamma:=\tilde{\gamma}\circ\phi$. By definition $\gamma_{\vert SL(2)_1T_r}$ is trivial and 
$\left(\begin{smallmatrix}
-1 & 0 \\
0  & -1
\end{smallmatrix}\right)$
acts trivially on $X(i,g,l)\otimes_kk_\gamma$.  Consequently, the action of $SL(2)_1T_rH$ on $X(i,g,l)\otimes_kk_\gamma$ factors through $\pi$.
Therefore we can lift the $SL(2)_1T_rH$-module $X(i,g,l)$ to a $\G$-module $Y(i,g,l)$ for all $i\in\{0,\dots,p-2\}$.\\
In the same way we can twist the Weyl modules $V(d)$ and obtain a module $\tilde{V}(d)$ for any 
amalgamated polyhedral group scheme.\\

 As before, let $C\subseteq SL(2,k)$
be a set of representatives for $SL(2,k)/B$ with $1,w_0\in C$. Let $\G$ be a group scheme and $\cN$ be a
normal subgroup scheme of $\G$. For an $\cN$-module $Y$ we denote by $\G_Y$ the stabilizer of $Y$.
In our situation, we will have $\G^0\subseteq \cN$, so that $\G_Y$ is the unique subgroup scheme of $\G$
with $\G^0\subseteq\G_Y$ and $\G_Y(k)=\G(k)_Y$. In the same way, for $h\in\G(k)$, the conjugated subgroup
scheme $\G_Y^h$ is determined by $\G(k)_Y^h$.
\begin{prop}
Let $\G$ be an amalgamated polyhedral group scheme and $r$ be the height of $\G^0$.
Then the following statements hold:
\begin{enumerate}[(a)]
\item   For all $l\in\N$, $i\in\{0,\dots,p-2\}$ and $g\in C$ there is an $h\in SL(2)(k)$ such that the group scheme $\G_{Y(i,g,l)}^h$ is either equal to
		$P\cS\cQ_{(p^r)}$ or to $P\cS\cC_{(np^r)}$ for one $n\in\N$ with $(n,p)=1$. 
\item	If $\G_{Y(i,g,l)}^h$ is equal to $P\cS\cQ_{(p^r)}$, then $hg\in h_4B\cup h_4w_0B$.
\item	If $\G_{Y(i,g,l)}^h$ is equal to $P\cS\cC_{(np^r)}$ and $n>1$, then $hg\in B\cup w_0B$. 
\item 	$Y(i,g,l)$ is a $\G_{Y(i,g,l)}$-module.
\end{enumerate}
\end{prop}
\begin{proof}
\begin{enumerate}[(a)]
\item The variety $\P(V_{\sl2})\cong\P^1$ is smooth. Let $G:=\G(k)$.
Thanks to \cite[4.3.2]{farnsteiner2012extensions} and \cite[4.1.3]{farnsteiner2012extensions}, the subgroup scheme $\G_{Y(i,g,l)}$ is
isomorphic to an amalgamated polyhedral group scheme. In particular, by \cite[3.3]{farnsteiner2006polyhedral} this isomorphism is given
via conjugation by an element $h\in SL(2)(k)$.
Due to \ref{inertia group cyclic for tubes}, the stabilizer $\G_{Y(i,g,l)}(k)=G_{Y(i,g,l)}$ is a cyclic group, so that we obtain (a) from \cite[4.3.2]{farnsteiner2012extensions} and the table in \cite[3.3]{farnsteiner2006polyhedral}.
\item A direct computation shows $\G_{Y(i,g,l)}^h=\G_{Y(i,hg,l)}$. Assume $\G_{Y(i,g,l)}=P\cS\cQ_{(p^r)}$. Let $\pi:SL(2)\rightarrow PSL(2)$ be the canonical projection and $\hat{G}\subseteq SL(2)(k)$ its preimage of
$G$. By definition of $P\cS\cQ_{(p^r)}$, the group $\hat{G}_{Y(i,g,l)}$ is a subgroup of $T^{h_4}$. Therefore $\hat{G}_{Y(i,g,l)}$ 
does only stabilize the points $[h_4.e]$ and $[h_4w_0.e]$ in $\P(V_{\sl2})$. 
By \ref{stabilizers equal}, we have $\hat{G}_{Y(i,g,l)}=\hat{G}_{[g.e]}$. Therefore we obtain 
$[g.e]\in\{[h_4.e],[h_4w_0.e]\}$ which yields the assertion.
\item Assume $\G_{Y(i,g,l)}=P\cS\cC_{(np^r)}$ with $n>1$ and $(n,p)=1$. 
By definition of $P\cS\cC_{(np^r)}$, the group $\hat{G}_{Y(i,g,l)}$ is a subgroup of $T$. The only points in $\P(V_{\sl2})$ which are stabilized by $T$
are $[e]$ and $[w_0.e]$. This yields as above $g\in B\cup w_0B$. 
\item First let $\G_{Y(i,hg,l)}=\G_{Y(i,g,l)}^h=P\cS\cC_{(np^r)}$ with $n>1$ and $(n,p)=1$. By (c), we obtain $hg\in B\cup w_0B$. Therefore
$T_{(2n)}\subseteq N_{SL(2)}(T)\cap B^{hg}$ and $Y(i,hg,l)$ is an $SL(2)_1T_{(2np^r)}$-module. As
$\left(\begin{smallmatrix}
-1 & 0 \\
0  & -1
\end{smallmatrix}\right)$ acts trivially on $Y(i,hg,l)$, it is also a $P\cS\cC_{(np^r)}$-module. \\
If $\G_{Y(i,hg,l)}$ is equal to $P\cS\cQ_{(p^r)}$, we obtain $hg\in h_4B\cup h_4w_0B$ by (b).
Hence $\cH_4\subseteq N_{SL(2)}(T)\cap B^{hg}$. As $\pi(SL(2)_1T_r\cH_4)=P\cS\cQ_{(p^r)}$,
we obtain that $Y(i,hg,l)$ is a $P\cS\cQ_{(p^r)}$-module.\\
In both cases conjugation by $h^{-1}$ yields that $Y(i,g,l)$ is a $\G_{Y(i,g,l)}$-module.
\end{enumerate}
\end{proof}
Let $\G$ be an amalgamated polyhedral group scheme. Then $\G(k)$ acts on $X:=SL(2)/B$ via left multiplication of the preimage of the projection
$SL(2)\rightarrow PSL(2)$. We denote by $C_\G$ a set of representatives for $X/\G(k)$ with $1\in C_\G$ and $w_0\in C_\G$, if $w_0$
is not in the same orbit as $1$. \\
We say that an extension $A:B$ of $k$-algebras is separable if the multiplication $A\otimes_BA\rightarrow A$ is a split surjective homomorphism
of $(A,A)$-bimodules. Using basic properties of separable extensions (\cite[10.8]{pierce1982associative}), we obtain for every indecomposable $A$-module $M$
an indecomposable direct summand $N$ of the $B$-module $\res_B^AM$ such that $M$ is a direct summand of $\ind_B^AN$. 
\begin{thm}\label{modules in tubes}
Let $\G$ be a amalgamated polyhedral group scheme.
Let $M$ be an indecomposable $\G$-module of complexity $1$. 
Then there are unique $l\in\N$, $i\in\{0,\dots,p-2\}$ and $g\in C_\G$ such that $M$ is isomorphic to $\ind_{\G_{Y(i,g,l)}}^\G (Y(i,g,l)\otimes_kk_\alpha)$ for a character $\alpha\in X(\G_{Y(i,g,l)})$.
Moreover, the character $\alpha$ can be chosen as a unique element of a subgroup of $X(\G_{Y(i,g,l)})$ determined by the following cases:
 \begin{enumerate}[(a)]
 \item If $g\in\{1,w_0\}$ then $\alpha\in X(\G_{Y(i,g,l)}/\G_1)$.
 \item If $g\not\in\{1,w_0\}$ then $\alpha\in X(\G_{Y(i,g,l)}/\G^0)$.
 \end{enumerate}
\end{thm}
\begin{proof}
The group algebra $k\G$ is isomorphic to $k\G^0\#k\G(k)$ and thanks to \cite[6.2.1]{farnsteiner2006polyhedral} the group
$\G(k)$ is linearly reductive. By \cite[3.15]{doi1989hopf}, the extension $k\G:k\G^0$ is separable, so that
there is an indecomposable direct summand $N$ of $\res_{\G^0}^\G M$ such that $M$ is a direct
summand of $\ind^\G_{\G^0} N$. 
General properties of the complexity (\cite[II.2]{shanghai}) yield
\[1=\cx_{k\G}(M)\leq\cx_{k\G}(\ind_{\G^0}^\G N)\leq\cx_{k\G^0}(N)\leq\cx_{k\G^0}(M)\leq\cx_{k\G}(M)=1.\]
Hence $N$ has also complexity $1$ and consequently there is
$g\in C$ with $V_{\sl2}(N)=\{g.[e]\}$. We first assume that $g\in\{1,w_0\}$. Then there are unique 
$l\in\N$ and $i\in\{0,\dots,p-2\}$ with $\res_{\G_1}^{\G^0}N\cong g.W(lp+i)$.
Moreover, as $\G^0/\G_1$ is linearly reductive, \cite[3.15]{doi1989hopf} yields that $k\G^0:k\G_1$ is separable.
Therefore, $N$ is isomorphic to a direct summand of $\ind_{\G_1}^{\G^0} g.W(lp+i)$. As $p$ does not divide $\dim_kY(i,g,l)$, 
\cite[2.1.4]{farnsteiner2009group} shows that $Y(i,g,l)$ is $\G^0/\G_1$-regular.
As $\G_{Y(i,g,l)}\cong P\cS\cC_{(np^r)}$ we have that $k(\G_{Y(i,g,l)}/\G_1)=\bigoplus_{\beta\in X(\G_{Y(i,g,l)}/\G_1}k_\beta$ is
a decomposition into simple $\G_{Y(i,g,l)}/\G_1$-modules. Hence an application of \ref{decomp of induced stable modules}
yields a decomposition \[\ind_{\G_1}^{\G_{Y(i,g,l)}}Y(i,g,l)\cong\bigoplus_{\beta\in X(\G_{Y(i,g,l)}/\G_1)}Y(i,g,l)\otimes_kk_\beta.\]
Since $\res_{\G_1}^{\G_{Y(i,g,l)}}Y(i,g,l)$ is isomorphic to $g.W(lp+i)$ the
module $\ind_{\G^0}^{\G_{Y(i,g,l)}}N$ is a direct summand of $\ind_{\G_1}^{\G_{Y(i,g,l)}}Y(i,g,l)$.
The group algebra $k\G$ is isomorphic to the skew group algebra $k\G^0\#k\G(k)$.  Therefore an application of \cite[4.5.2]{karpilovsky1987algebraic}
yields that the indecomposable direct summands of $\ind_{\G^0}^\G N$ are isomorphic to $\ind_{\G_{Y(i,g,l)}}^\G (Y(i,g,l)\otimes_kk_\alpha)$ with
a unique $\alpha\in X(\G_{Y(i,g,l)}/\G_1)$.\\
If $g\not\in\{1,w_0\}$, then there are unique $l\in\N$ and $i\in\{0,\dots,p-2\}$ with $N\cong X(i,g,l)$. Similar arguments as above now yield the second assertion. \\
Now let $h\in C$ with $M\cong \ind_{\G_{Y(i,h,l)}}^\G (Y(j,h,m)\otimes_kk_\beta)$ for $\beta$ in 
$X(\G_{Y(j,h,m)}/\G_1)$ respectively $X(\G_{Y(j,h,m)}/\G^0)$, $m\in\N$ and $j\in\{0,\dots,p-2\}$.
Then
\[\res_{\G_1}^\G M\cong \bigoplus_{u\in \G(k)}ug.W(lp+i)\cong \bigoplus_{v\in \G(k)}vh.W(mp+j).\]
Thanks to \ref{X(i,g,l)}, we obtain $m=l$, $j=i$ and there is $v\in\G(k)$ such that $vh$ is represented in $C$ by $g$. Therefore $h$ and $g$ are in the same $\G(k)$-orbit and we obtain
a unique element $\tilde{g}\in C_\G$ with $M\cong\ind_{\G_{Y(i,\tilde{g},l)}}^\G (Y(i,\tilde{g},l)\otimes_kk_\alpha)$.
\end{proof}

\begin{thm}\label{modules in euclidean components}
Let $\G$ be a finite subgroup scheme of $PSL(2)$ with $\G_1\cong PSL(2)_1$ and tame principal block.
Let $M$ be an indecomposable $\G$-module of complexity $2$. 
Then $M$ is isomorphic to one of the following pairwise non-isomorphic $\G$-modules:
\begin{enumerate}
\item $\tilde{V}(d)\otimes_kS$ or $M\cong \tilde{V}(d)^*\otimes_kS$ for $d\geq p$ with $d\not\equiv-1\;(\mod p)$ and a simple $\G/\G_1$-module $S$ 
\item $\tilde{V}(i)\otimes_kS$ with $0\leq i\leq p-1$ and a simple $\G/\G_1$-module $S$.
\end{enumerate}
\end{thm}
\begin{proof}
Since $\G/\G_1$ is linearly reductive, \cite[3.15]{doi1989hopf} yields that
the extension $k\G:k\G_1$ is separable, so that there is 
an indecomposable direct summand $N$ of $\res_{\G_1}^\G M$ such that $M$ is a direct summand of $\ind^\G_{\G_1} N$. 
As in the proof of \ref{modules in tubes}, we obtain $\cx_{\G^0}(N)=\cx_{\G}(M)=2$.
Since $\G_1$ is isomorphic to $PSL(2)_1\cong SL(2)_1$, we obtain that $N$ is isomorphic to 
$V(d)$, $V(d)^*$ or $V(i)$ for a unique $d\geq p$ with $d\not\equiv-1\;(\mod p)$ or a unique $0\leq i\leq p-1$.
As noted at the beginning of this section these modules are the restrictions of the $\G$-modules 
$\tilde{V}(d)$, $\tilde{V}(d)^*$ and $\tilde{V}(i)$.
Hence we can assume that $N$ is a $\G$-module which is isomorphic to one of these modules.
By \cite[2.1.4]{farnsteiner2009group}, the module $N$
is $\G^0/\G_1$-regular and \ref{decomp of induced stable modules} yields that the indecomposable direct summands of $\ind_{\G_1}^\G N$ are of the
form $N\otimes_kS$ for a unique $\G/\G_1$-module $S$.
\end{proof}
\begin{rmk}
	\begin{enumerate}
\item The modules in \ref{modules in euclidean components} are exactly the modules which appear in Euclidean components
of the Auslander-Reiten quiver of $\G$. The Auslander-Reiten quiver of $\G$ is classified by \cite[7.4]{farnsteiner2006polyhedral}
and the tree types of the Euclidean components coincide with the McKay-quiver of $\G/\G_1$.
This fact leads to a direct connection between these quivers which will be discussed in a separate paper.
\item	The modules listed in (2) are exactly the simple $\G$-modules.
\end{enumerate}
\end{rmk}
With these results in hand, one is able to write down the classification of modules 
for the amalgamated polyhedral group schemes. Moreover, we are able to assign to each module its component in the Auslander-Reiten quiver
by using the results in this work.
As an example, we will work out the representation theory of the amalgamated dihedral group schemes.\\
Let $m\geq 2$, $m=np^r$ with $(n,p)=1$.
By \cite[7.4]{farnsteiner2006polyhedral}, the stable Auslander-Reiten quiver of the $\frac{p-1}{2}$ non-simple blocks of
$kP\cS\cQ_{(m)}$ consists of two Euclidean components of type $\Z[\tilde{D}_{np^{r-1}+2}]$, two exceptional tubes
$\Z[A_\infty]/(\tau^{np^{r-1}})$, four exceptional tubes
$\Z[A_\infty]/(\tau^2)$ and infinitely many homogeneous tubes. \\
The representation theory of the amalgamated dihedral group schemes is the most complicated case.
For $n=2$ and $r=1$ we are not able to distinguish the two cases of exceptional tubes and all exceptional tubes are of rank $2$.
Hence, for simplification, we will not consider this situation, as it has to be handled a little bit differently.
\begin{prop}
Let $m\geq 2$ with $m=p^rn$ and $(n,p)=1$. Assume $r>1$ or $n>2$. Let $\G=P\cS\cQ_{(m)}$ and assume
$\{1,h_4,h_4w_0\}\subseteq C_\G$. Let $M$ be an indecomposable non-projective $\G$-module.
Then $M$ is isomorphic to one of the modules belonging to the following list of pairwise non-isomorphic $\G$-modules:
\begin{enumerate}[(i)]
\item $\tilde{V}(d)\otimes_kS$, $\tilde{V}(d)^*\otimes_kS$, $\tilde{V}(i)\otimes_kS$ for $d\geq p$, $d\not\equiv -1\mod p$,
		 $S$ a simple $\G/\G_1$-module and and $0\leq i\leq p-1$.
		(Modules belonging to Euclidean components)
\item $\ind_{P\cS\cC_{(m)}}^\G Y(i,1,l)\otimes_kk_\lambda$ for $i\in\{0,\dots p-2\}$, $l\in\N$ and $\lambda\in X(\mu_{(np^{r-1})})$.
		(Modules belonging to exceptional tubes of rank $np^{r-1}$)
\item $\ind_{P\cS\cQ_{(p^r)}}^\G Y(i,h_4w_0^j,l)\otimes_kk_\lambda$ for $j\in\{0,1\}$, $i\in\{0,\dots p-2\}$, $l\in\N$ and $\lambda\in X(\mu_{(2)})$.
		(Modules belonging to exceptional tubes of rank $2$)
\item $\ind_{P\cS\cC_{(p^r)}}^\G Y(i,g,l)$ for $g\in C_\G\setminus\{1,h_4,h_4w_0\}$, 
$i\in\{0,\dots p-2\}$ and $l\in\N$.
				(Modules belonging to homogeneous tubes)
\end{enumerate}
\end{prop}
\begin{proof}
This proof has two parts. The first part uses the results of this section to determine the modules as listed above.
In the second part we will use the results about the Auslander-Reiten quiver and stabilizers to assign these modules to
their components in the Auslander-Reiten quiver.
\begin{enumerate}
\item
If $M$ has complexity $2$ it is by \ref{modules in euclidean components} isomorphic to one of the modules in (i). Hence 
we can assume that $M$ has complexity $1$.
Then an application of \ref{modules in tubes} yields unique $l\in\N$, $i\in\{0,\dots,p-2\}$ and $g\in C_\G$ such that $M$ is isomorphic to 
$\ind_{\G_{Y(i,g,l)}}^\G (Y(i,g,l)\otimes_kk_\alpha)$ for a unique character $\alpha$.\\
Let $G:=\G(k)$. By \ref{stabilizers equal} we know $G_{Y(i,g,l)}=G_{[g.e]}$. By computing the stabilizers for the action of $G$
on $\P(V_{\sl2})$ we obtain the following cases:
\begin{enumerate}
\item
If $g=1$, then $\G_{Y(i,1,l)}=P\cS\cC_{(m)}$ and $\alpha\in X(\G_{Y(i,1,l)}/\G_1)\cong X(\mu_{(np^{r-1})})$.\\
\item
If $g\in\{h_4,h_4w_0\}$, then $\G_{Y(i,g,l)}= P\cS\cQ_{(p^r)}$ and $\alpha\in X(\G_{Y(i,g,l)}/\G^0)\cong X(\mu_{(2)})$.\\
\item
Otherwise, $\G_{Y(i,g,l)}= P\cS\cC_{(p^r)}$ and $\alpha\in X(\G_{Y(i,g,l)}/\G^0)\cong X(\mu_{(1)})=\{1\}$.\\
\end{enumerate}
These are exactly the cases (ii)-(iv).
\item
If $M$ has complexity $2$ then by \cite[3.3]{farnsteiner2007support} the component of $M$ is not a tube and therefore Euclidean. \\
Let $M$ be of complexity $1$ and $g\in C_\G$ be arbitrary. Thanks to \cite[4.5.10]{marcus1999representation}, the component $\Theta\subseteq\Gamma_s(\G_{Y(i,g,l)})$ 
of the $\G_{Y(i,g,l)}$-module $Y(i,g,l)\otimes_kk_\alpha$ is isomorphic to the component of $M$. As $M$ has complexity $1$,
an application of \cite[3.3]{farnsteiner2007support} yields that the component $\Theta$ is a tube.
Let $\Xi\subseteq\Gamma_s(\G^0)$ be the component of $\res_{\G^0}^{\G_{Y(i,g,l)}}Y(i,g,l)$. Then $\Xi$ is a tube of rank $s\in\{1,p^{r-1}\}$. Denote by
$q$ the rank of $\Theta$.  Due to \ref{multiplicity of tubes for induction with cyclic group}, we obtain $q\leq \vert G_{Y(i,g,l)}\vert s$.
An application of \ref{inertia group cyclic for tubes} yields that $\vert G_{Y(i,g,l)}\vert= \vert G_{[g.e]}\vert\leq n$.\\
Assume that $\Theta$ is a tube of rank $np^{r-1}$. Then $np^{r-1}\leq \vert G_{Y(i,g,l)}\vert s\leq np^{r-1}$ and therefore $\vert G_{Y(i,g,l)}\vert=n$ and $s=p^{r-1}$.
If $r>1$, then $V_{\sl2}(\Xi)\in\{ke,kf\}$ and therefore $g\in\{1,w_0\}$. If $n>2$, the only points in $\P(V_{\sl2})$ stabilized by $G_{Y(i,g,l)}=G_{[g.e]}$
are $[e]$ and $[w_0.e]$. Hence we have in both cases $g=1$, as $w_0$ is in the same $G$-orbit as $1$.
The quasi-length of $\res_{\G^0}^{\G_{Y(i,g,l)}}Y(i,g,l)$ in $\Xi$ is $l$. By \ref{multiplicity of tubes for induction with cyclic group},
the morphism $\res_{\G^0}^{\G_{Y(i,g,l)}}:\Theta\rightarrow\Xi$
preserves the quasi-length, so that the module $Y(i,g,l)\otimes_kk_\beta$ has also quasi-length $l$. As 
the number of  $\G$-modules of quasi-length $l$ belonging to exceptional tubes of rank $np^{r-1}$
is $(p-1)np^{r-1}$ and this number equals the number of
modules in (ii) for a fixed $l$, all these modules belong to an exceptional tube of rank $np^{r-1}$.\\
Now assume that $\Theta$ is a tube of rank $2$. Then $g\neq1$ and therefore $s=1$ and $\vert G_{Y(i,g,l)}\vert\leq 2$.
Since $2\leq\vert G_{Y(i,g,l)}\vert s=\vert G_{Y(i,g,l)}\vert$ we obtain $\vert G_{Y(i,g,l)}\vert=2$. Therefore $M$ is isomorphic to a module in (iii).
With the same arguments as above, all modules in (iii) belong to an exceptional tube of rank $2$. \\
Now the remaining modules in (iv) have to belong to homogeneous tubes.
\end{enumerate}
\end{proof}
\begin{rmk} The representation theory of the remaining amalgamated polyhedral group schemes can be analyzed
		with these methods in a similar way.
\end{rmk}

%% file: acknowledgement.tex
\subsubsection*{Acknowledgement}
The results of this article are part of my doctoral thesis, which I am currently writing at the University
of Kiel. I would like to thank my advisor Rolf Farnsteiner for his continuous support as well as for helpful remarks.
Furthermore, I thank the members of my working group for proofreading.

\subsubsection*{Funding}
This work was supported by the D.F.G. priority program SPP1388 'Darstellungstheorie'.